\documentclass[11pt,leqno]{amsart}

\usepackage[top=2.4 cm,bottom=1.9cm,left=2.5 cm,right=2 cm]{geometry}
\usepackage[T1]{fontenc}

\usepackage{bbm}
\usepackage{esint}
\usepackage{graphicx}
\usepackage{color}
\usepackage{amsfonts,amsmath,amssymb,mathrsfs}
\usepackage{psfrag}
\usepackage{enumitem}
\newcounter{stepnb}

\usepackage{hyperref}\hypersetup{colorlinks=true,linkcolor=blue,citecolor=blue}

\newtheorem{theorem}{Theorem}
\newtheorem{lemma}[theorem]{Lemma}
\newtheorem{proposition}[theorem]{Proposition}

\theoremstyle{plain} \newtheorem{corollary}[theorem]{Corollary}

\newtheorem{definition}{Definition}[section]
\theoremstyle{remark}

\allowdisplaybreaks
\theoremstyle{plain} 

\newcommand{\N}{\mathbb{N}}

\newcommand{\R}{\mathbb{R}}
\newcommand{\Q}{\mathbb{Q}}
\newcommand{\M}{{\mathscr M}}

\newcommand{\G}{{\mathcal G}}

\newcommand{\Id}{{\mathrm I}}

\newcommand{\AG}{\mathrm{AG}}
\renewcommand{\H}{{\mathscr H}}
\renewcommand{\L}{{\mathscr L}}

\renewcommand{\S}{\mathbb{S}}
\renewcommand{\P}{\mathcal{P}}

\DeclareMathOperator{\BV}{BV}

\newcommand{\e}{\varepsilon}

\newcommand{\TV}{\text{\rm TotVar}}

\newcommand{\be}{\begin{equation}}
\newcommand{\eq}{\end{equation}}

\renewcommand{\div}{{\rm div}\,}

\newcommand{\loc}{\mathrm{loc}}

\newcommand{\ac}{\mathrm{ac}}
\newcommand{\n}{\mathbf{n}}

\title{The rectifiability of the entropy defect measure for Burgers equation}
\author[E.~Marconi]{Elio Marconi}
\address{E.M. Departement Mathematik und Informatik,
Universit\"at Basel, Spiegelgasse 1, CH-4051 Basel, Switzerland.}
\email{elio.marconi@unibas.ch}
\thanks{The author acknowledges ERC Starting Grant 676675 FLIRT and Xavier Lamy for several discussions on this topic.}

\begin{document}
\maketitle

\begin{abstract}
We consider bounded weak solutions to the Burgers equation for which every entropy dissipation is representable by a measure and we prove that all these measures
are concentrated on the graphs of countably many Lipschitz curves.
The main tool is the Lagrangian representation, which is an extension of the method of characteristics to the non-smooth setting.
\end{abstract}

\section{Introduction}
We consider the Cauchy problem for the Burgers equation:
\begin{equation}\label{E_Burgers}
\begin{cases}
&u_t + \left(\frac{u^2}{2}\right)_x=0, \qquad (t,x) \in [0,T]\times \R\\
&u(0,\cdot)=u_0,
\end{cases}
\end{equation}
with $T>0$ and $u_0 \in L^\infty(\R)$. It is known since the early stages \cite{Kruzhkov_contraction} of the theory of scalar conservation laws that the Cauchy problem \eqref{E_Burgers} is well-posed in the setting of bounded 
entropy solutions, namely weak solutions with the following additional constraint: 
for every smooth convex entropy $\eta:\R\to \R$ and relative flux $q:\R\to \R$ defined up to a constant by $q'(v)=\eta'(v)v$, it holds
\begin{equation*}
\mu_\eta:=\partial_t \eta(u) + \partial_x q(u) \le 0 
\end{equation*}
in the sense of distributions.
The celebrated Oleinik one-sided Lipschitz estimate \cite{Oleinik_translation} implies that bounded entropy solutions to \eqref{E_Burgers} belong to $\BV_\loc((0,T]\times \R)$ and the structure of the
solution is well described by means of the theory of $\BV$ functions.

In this paper we investigate the structure of more general weak solutions introduced in the following definition:
\begin{definition}
We say that $u$ is a bounded weak solutions to \eqref{E_Burgers} with \emph{finite entropy production} if $u \in C^0([0,T];L^1(\R)) \cap L^\infty([0,T]\times \R)$ and for every convex 
entropy $\eta$ and corresponding flux $q$
\begin{equation*}
\mu_\eta:= \partial_t \eta(u) + \partial_x q(u) \in \M([0,T]\times \R),
\end{equation*}
where $\M$ denotes the set of finite Radon measures.
\end{definition}
The interest toward these solutions is clearly motivated in \cite{LO_Burgers}.
We report the main motivations here for completeness: weak solutions with finite entropy production arise as domain of the $\Gamma$-limit of a sequence of functionals
introduced in \cite{bertini_al} to study large deviations principles for stochastic approximations of entropy solutions (see also \cite{Mariani_large}).
Moreover there is a strong analogy between the solutions introduced above and the weak solutions of the eikonal equation 
\begin{equation}\label{E_eikonal}
|\nabla \phi|=1
\end{equation}
in the plane arising as domain of the 
$\Gamma$-limit as $\e\to 0$ of the Aviles-Giga functionals:
\begin{equation*}
\AG_\e(\phi):= \int \left(\frac{\e}{2}|\nabla^2\phi|^2 + \frac{1}{2\e}(|\nabla \phi|^2-1)^2\right).
\end{equation*}
The same analogy holds for a very related model about thin ferro-magnetic films studied in \cite{RS_magnetism,RS_magnetism2}, see also \cite{Riviere_parois}.
At a formal level the link between conservation laws and the eikonal equation is the following (see for example \cite{DLO_JEMS}): 
setting $m=\nabla^\perp \phi$ it holds $\div \,m=0$. Moreover we can impose the constraint $|m|=1$ introducing a phase $\theta \in \S^1$ so that $m=(\cos \theta, \sin \theta)$ and 
the eikonal equation reduces to the $\S^1$-valued scalar conservation law
\begin{equation*}
\partial_{x_1}\cos \theta + \partial_{x_2}\sin \theta =0.
\end{equation*}
This analogy has been pushed much further when the notion of entropy has been transferred from the theory of conservation laws to the eikonal equation
\cite{JK_entropies,ADLM_eikonal,DMKO_compactness} and this allowed to provide a kinetic formulation for \eqref{E_eikonal}, see for example \cite{JP_kinetic}.
The kinetic formulation for scalar conservation laws has been introduced in \cite{LPT_kinetic} in the context of entropy solutions: in our setting it reads as follows 
(see for example \cite{DLW_structure}): given a bounded weak solution to \eqref{E_Burgers} with finite entropy production $u$ there exists a Radon measure 
$\mu \in \M([0,T]\times \R \times \R)$ such that
\begin{equation*}
\partial_t\chi + v \partial_x \chi = \partial_v \mu, \qquad \mbox{where} \qquad 
\chi (t,x,v):= 
\begin{cases} 
1 & \mbox{if }0 < v \le u(t,x), \\
-1 & \mbox{if }u(t,x)\le v <0, \\
0 & \mbox{otherwise}.
\end{cases}
\end{equation*}
The kinetic measure $\mu$ encodes all the entropy production measures $\mu_\eta$ by means of the following formula:
\begin{equation}\label{E_all_entropies}
\langle \mu_\eta, \phi\rangle = \int_{[0,T]\times \R \times \R} \eta''(v)\phi(t,x) d\mu.
\end{equation}
In particular we will consider the measure
\begin{equation}\label{E_nu}
\nu:= (p_{t,x})_\sharp |\mu| = \sup_{\|\eta''\|_{L^\infty}\le 1}\mu_\eta, 
\end{equation}
where $p_{t,x}:[0,T]\times \R^2\to [0,T]\times \R$ is the standard projection on the first two components and the wedge denotes the supremum operator in the set of measures.
The equality in \eqref{E_nu} follows from \eqref{E_all_entropies} (see for example \cite{LO_Burgers}).

In contrast with the case of entropy solutions, bounded weak solutions with finite entropy production are not locally $\BV$: 
they belong to $B^{1/3,3}_{\infty, \loc}$ \cite{GP_optimal} and this regularity is optimal  \cite{DLW_averaging}.
The same result has been established recently in the case of the eikonal equation \cite{GL_eikonal}.
Nevertheless these solutions share several fine properties with functions of bounded variation: in \cite{Lecumberry_PhD,DLW_structure} it has been proved that the set
\begin{equation}\label{E_J}
J := \left\{(t,x)\in [0,T]\times \R: \limsup_{r\to 0}\frac{\nu(B_r(x))}{r} >0\right\}
\end{equation}
is $\H^1$-rectifiable and it admits strong traces on both sides. In particular for every entropy $\eta$ it holds
\begin{equation}\label{E_rect}
\mu_\eta\llcorner J = \big((\eta(u^+)-\eta(u^-))\n_t + (q(u^+)-q(u^-))\n_x)\big)\H^1 \llcorner J,
\end{equation}
where $\n=(\n_t,\n_x)$ is the normal to $J$ and $u^\pm$ are the traces.
Moreover every point in $J^c$ is a point of vanishing mean oscillation of $u$.
Recently, it has been proved in \cite{LO_Burgers} that the set $S$ of non-Lebesgue points of $u$ has Hausdorff dimension at most 1, almost filling the gap with the case of BV solutions,
where we would have $\H^1(J\setminus S)=0$ (see also \cite{M_Lebesgue} for the same result for more general conservation laws and \cite{Silvestre_CL} for the case of entropy solutions, where every point in $J^c$ is actually a continuity point of $u$).
Another relevant property of BV solutions is that for every smooth entropy $\eta$ the measure $\mu_\eta$ is concentrated on $J$.
This result has been proved also for entropy solutions for more general conservation laws in one space dimensions \cite{DLR_dissipation, BM_structure}, for continuous entropy solutions in several space 
dimensions \cite{Silvestre_CL, BBM_multid} and for continuous solutions to general conservation laws in one space dimension \cite{Dafermos_continuous}.
In both settings of conservation laws and of the eikonal equation, the proof of this concentration property is considered as a main step to prove the $\Gamma$-convergence
of the families of functionals discussed above, see  \cite{bertini_al,Lecumberry_magnetic}.

The main result of this work establishes that this property holds for bounded weak solutions to the Burgers equation with finite entropy production:
\begin{theorem}\label{T_main}
Let $u$ be a bounded weak solution to \eqref{E_Burgers} with finite entropy production. 
Then there exist countably many Lipschitz curves $\gamma_i:[0,T]\to \R$ such that $\nu$ (and therefore every entropy dissipation measure $\mu_\eta$)
is concentrated on
\begin{equation*}
J':= \bigcup_{i\in \N}\{(t,x)\in [0,T]\times \R: x=\gamma_i(t)\}.
\end{equation*}
\end{theorem}

We observe that by definition of $J$ it follows that $\nu\llcorner J^c$ does not charge any set with finite $\H^1$-measure, and therefore any 1-rectifiable set; 
in particular, in the statement of Theorem \ref{T_main}, we can replace the set $J'$ with the set $J$ defined in \eqref{E_J} and recover the representation \eqref{E_rect} for 
$\mu_\eta=\mu_\eta\llcorner J$.

The fine properties of solutions to \eqref{E_Burgers} and \eqref{E_eikonal} that we discussed and the regularizing effect for conservation laws can be attributed to 
the characteristics structure that these equations own.
The main tool to get Theorem \ref{T_main} is the Lagrangian representation, which is an extension of the method of characteristics to deal with nonsmooth solutions.
It has been introduced first for entropy solutions to general scalar conservation laws in \cite{BBM_multid}, relying on the transport-collapse scheme by Brenier \cite{Brenier_TC}.
Then this notion has been extended to cover the case of bounded weak solutions with finite entropy production in \cite{M_Lebesgue}, building on the kinetic formulation introduced in \cite{DLW_structure}:
in the particular case of the Burgers equation it takes the form provided in Definition \ref{D_Lagr}.
We would like to point out that the notion of Lagrangian representation is inspired by the superposition principle for nonnegative measure-valued solutions
to the continuity equation (see \cite{AC_superposition}). In particular it shares with it the main feature that the evolution of the solution $u$ is described as the result of 
the evolutions of the single particles along the characteristics.
Finally we mention that the notion of Lagrangian representation is available for solutions with finite entropy production to the eikonal equation too \cite{M_eikonal}.

As a consequence of Theorem \ref{T_main} and \eqref{E_rect} we obtain the following result:
\begin{theorem}\label{T_one_entropy}
Let $u$ be a bounded weak solution to the Burgers equation with finite entropy production. Then it holds
\begin{equation*}
|(p_{t,x})_\sharp \mu| = (p_{t,x})_\sharp |\mu|.
\end{equation*}
In particular, denoting by $\bar \eta(u)=u^2/2$,  if $\mu_{\bar \eta}\le 0$ then $\mu\le 0$, namely if $u$ dissipates the quadratic entropy $\bar \eta$, then $u$ is the entropy solution to \eqref{E_Burgers}.
\end{theorem}
The second part of this statement is known even under milder assumptions: in \cite{Panov_one_entropy} the result has been proved for all bounded weak solutions to \eqref{E_Burgers} and
in \cite{DLOW_one_entropy} it has been extended to weak solutions in $L^4_\loc$.
Both proofs rely on the link between entropy solutions to \eqref{E_Burgers} and viscosity solutions to the Hamilton-Jacobi equation
\begin{equation*}
v_t + \frac{v_x^2}{2}=0.
\end{equation*}
The interest in looking for alternative proofs is motivated by the problem of extending this result to systems of conservation laws, where the link with the Hamilton-Jacobi equation
is not available: see \cite{KV_one_entropy} for a recent result in this direction.

We finally observe that both Theorem \ref{T_main} and Theorem \ref{T_one_entropy} can be proved with the same strategy and minor modifications in the case of bounded weak solutions with finite entropy production to conservation laws with uniformly convex fluxes.

\section{Lagrangian representation}
In this section we present the notion of Lagrangian representation introduced in \cite{M_Lebesgue} and we discuss some of its properties.
Since we consider bounded solutions we can assume without loss of generality that $u$ takes values in $[0,1]$.

For every function $f:\R\to [0,1]$ we denote its hypograph and its epigraph by
\begin{equation*}
E_f:= \{(x,v)\in \R\times [0,1]: v\le f(x)\} \qquad \mbox{and} \qquad E_f^c:= \{(x,v)\in \R\times [0,1]: v\ge f(x)\}
\end{equation*}
respectively.
Moreover we denote by
\begin{equation*}
\Gamma :=\left\{\gamma=(\gamma_x,\gamma_v)\in BV([0,T];\R\times [0,1]): \gamma_x \mbox{ is Lipschitz}
\right\}.
\end{equation*}
It will be useful to consider the standard decomposition of the measure $Df \in \M(\R)$, where $f\in \BV(\R,\R)$ (see for example \cite{AFP_book}).
We will adopt the following notation:
\begin{equation}
Df= D^{\ac}f + D^c f + D^jf = \tilde Df + D^jf,
\end{equation}
where $D^{\ac}f$, $D^cf$ and $D^jf$ denote the absolutely continuous part, the Cantor part and the atomic part of $Df$ respectively; we refer to $\tilde Df$ as the diffuse part of $Df$.

\begin{definition}\label{D_Lagr}
Let $u$ be a weak solution to \eqref{E_Burgers} with finite entropy production. We say that the Radon measure $\omega_h \in \M(\Gamma)$ is a \emph{Lagrangian representation} of the hypograph of $u$ if the following conditions hold:
\begin{enumerate}
\item for every $t\in [0,T)$ it holds
\begin{equation}\label{E_repr_formula}
(e_t)_\sharp \omega_h = \mathscr L^{2}\llcorner E_{u(t)},
\end{equation}
where $e_t$ denotes the evaluation map defined by
\begin{equation}\label{E_evaluation}
\begin{split}
e_t : \Gamma &\to \R. \\
\gamma &\mapsto \gamma(t)
\end{split}
\end{equation}
\item the measure $\omega_h$ is concentrated on the set of curves $\gamma\in \Gamma$ such that
\begin{equation}\label{E_characteristic}
\dot\gamma_x(t)=\gamma_v(t) \quad \mbox{for a.e. }t\in [0,T);
\end{equation}
\item it holds the integral bound
\begin{equation}\label{E_reg}
\int_\Gamma \TV_{[0,T)} \gamma_v d\omega_h(\gamma) <\infty.
\end{equation}
\end{enumerate}
Similarly we say that $\omega_e \in \M(\Gamma)$ is a \emph{Lagrangian representation} of the epigraph of $u$ if Conditions (2) and (3) hold and (1) is replaced by
\begin{equation}\label{E_repr_e}
(e_t)_\sharp \omega_e = \mathscr L^2\llcorner E^c_{u(t)} \qquad \mbox{for every }t \in [0,T].
\end{equation}
\end{definition}

A useful property of these representations is that the kinetic measure (and therefore any entropy dissipation measure $\mu_\eta$) can be decomposed along the characteristics.

Given $\gamma \in \Gamma$ we consider
\begin{equation*}
\mu_\gamma=(\Id, \gamma)_\sharp  \tilde D_t \gamma_v + \H^1\llcorner E_{\gamma}^+ -\H^1\llcorner E_{\gamma}^- \in \M([0,T]\times \R \times [0,1]),
\end{equation*}
where
\begin{equation*}
\begin{split}
E_\gamma^+:=&\{(t,x,v): \gamma_x(t)=x, \gamma_v(t-)<\gamma_v(t+), v \in (\gamma_v(t-),\gamma_v(t+)) \}, \\
E_\gamma^-:=&\{(t,x,v): \gamma_x(t)=x, \gamma_v(t+)<\gamma_v(t-), v \in (\gamma_v(t+),\gamma_v(t-)) \},
\end{split}
\end{equation*}
$\Id:[0,T]\to [0,T]$ denotes the identity and $\tilde D_t \gamma_v$ denotes the diffuse part of the measure $D_t\gamma_v$.


\begin{proposition}\label{P_Lagrangian}
Let $u$ be a weak solution to \eqref{E_Burgers} with finite entropy production. 
Then there exist $\omega_h, \omega_e$ Lagrangian representations of the hypograph and of the epigraph of $u$ respectively enjoying the additional properties:
\begin{equation}\label{E_decomposition_mu}
\int_{\Gamma}\mu_\gamma d\omega_h(\gamma)=\mu = -\int_{\Gamma}\mu_\gamma d\omega_e(\gamma),
\end{equation}
\begin{equation}\label{E_variation}
\int_{\Gamma}|\mu_\gamma| d\omega_h(\gamma)=|\mu| = \int_{\Gamma}|\mu_\gamma| d\omega_e(\gamma).
\end{equation}
\end{proposition}

Eq. \eqref{E_decomposition_mu} asserts that $\mu$ can be decomposed along characteristics and Eq. \eqref{E_variation} says that it can be done minimizing
\begin{equation*}
\int_\Gamma \TV_{[0,T)} \gamma^2 d\omega_h(\gamma) \qquad \mbox{and} \qquad \int_\Gamma \TV_{[0,T)} \gamma^2 d\omega_e(\gamma). 
\end{equation*}
Moreover it follows from \eqref{E_decomposition_mu} and \eqref{E_variation} that we can separately represent the negative and the positive parts of $\mu$:
\begin{equation}\label{E_mu-}
\int_{\Gamma}\mu^-_\gamma d\omega_h(\gamma)=\mu^- = \int_{\Gamma}\mu^+_\gamma d\omega_e(\gamma) \qquad \mbox{and} \qquad 
\int_{\Gamma}\mu^+_\gamma d\omega_h(\gamma)=\mu^+ = \int_{\Gamma}\mu^-_\gamma d\omega_e(\gamma).
\end{equation} 

The assertion of Proposition \ref{P_Lagrangian} regarding $\omega_h$ is proved in \cite{M_Lebesgue} for more general conservation laws and it is straightforward to adapt the 
same argument to get the existence of an $\omega_e$ as in the statement.
Alternatively, in order to prove the part concerning $\omega_e$, we can consider a Lagrangian representation $\tilde \omega_h$ of the hypograph of $\tilde u = 1-u$, which is a 
weak solution with finite entropy production to
\begin{equation*}
\tilde u_t + g(\tilde u)_x=0, \qquad \mbox{with} \qquad g(z)=-\frac{(z-1)^2}{2}.
\end{equation*}
Let $T:\Gamma \to \Gamma$ be defined by
\begin{equation*}
T(\gamma)(t)=(\gamma_x(t),1-\gamma_v(t)).
\end{equation*}
Then it is straightforward to check that the measure $\omega_e:=T_\sharp \tilde \omega_h$ satisfies the requirements in Proposition \ref{P_Lagrangian}.

The following lemma is an application of Tonelli theorem:
\begin{lemma}
For $\omega_h$-a.e. $\gamma \in \Gamma$ it holds that for $\L^1$-a.e. $t\in [0,T]$
\begin{enumerate}
\item $(t,\gamma_x(t))$ is a Lebesgue point of $u$;
\item $\gamma_v(t)< u(t,\gamma_x(t))$.
\end{enumerate}
We denote by $\Gamma_h$ the set of curves $\gamma\in \Gamma$ such that the two properties above hold.
Similarly for $\omega_e$-a.e. $\gamma \in \Gamma$ it holds that for $\L^1$-a.e. $t\in [0,T]$
\begin{enumerate}
\item $(t,\gamma_x(t))$ is a Lebesgue point of $u$;
\item $\gamma_v(t)> u(t,\gamma_x(t))$
\end{enumerate}
and we denote the set of these curves by $\Gamma_e$.
\end{lemma}
\begin{proof}
Let us prove the properties about $\Gamma_h$.
We denote by $S \subset [0,T]\times \R$ the set of non-Lebesgue points of $u$ and for every $t\in [0,T]$ we set
\begin{equation*}
\begin{split}
e_t^x : \Gamma &\to \R. \\
\gamma &\mapsto \gamma_x(t)
\end{split}
\end{equation*}
By \eqref{E_repr_formula} it follows that for every $t\in [0,T]$ it holds $(e_t^x)_\sharp \omega_h \le \L^1$. Since $\L^2(S)=0$ it holds
\begin{equation}\label{E_Ton1}
\begin{split}
 0 =&~ \left(\L^1 \otimes (e_t^x)_\sharp \omega_h \right)(S) \\
   = &~  \int_0^T  \omega_h(\{\gamma\in \Gamma: (t,\gamma_x(t))\in S\})dt \\
   = &~ \int_\Gamma \L^1(\{t\in [0,T]: (t,\gamma_x(t)) \in S\})d\omega_h(\gamma),
\end{split}
\end{equation} 
where the last equality follows by Tonelli theorem.
Similarly by \eqref{E_repr_formula} it follows that for every $t \in[0,T]$
\begin{equation}
\omega_h(\{\gamma \in \Gamma: \gamma_v(t)\ge u(t,\gamma_x(t))\})= (e_t)_\sharp \omega_h(\{(x,v)\in \R\times [0,1]: v \ge u(t,x)\})=0,
\end{equation}
where $e_t$ is defined in \eqref{E_evaluation}.
Therefore by Tonelli theorem it holds
\begin{equation}\label{E_Ton2}
\begin{split}
0 = &~\int_0^T \omega_h(\{\gamma \in \Gamma: \gamma_v(t)\ge u(t,\gamma_x(t))\})dt \\
= &~ \int_\Gamma \L^1(\{t \in [0,T]: \gamma_v(t) \ge u(t,\gamma_x(t))\})d\omega_h(\gamma).
\end{split}
\end{equation}
By \eqref{E_Ton1} and \eqref{E_Ton2} it follows that $\omega_h$ is concentrated on $\Gamma_h$.
The case of $\Gamma_e$ is analogous. 
\end{proof}

The following result is a well-known property of traces on Lipschitz curves:
\begin{lemma}\label{L_trace}
Assume that $\bar \gamma_x:(t^-,t^+)\subset [0,T]\to \R$ is a Lipschitz curve and that for $\L^1$-a.e. $t \in (t^-,t^+)$ the point $(t,\bar \gamma_x(t))$ is a Lebesgue point of $u \in L^\infty([0,T]\times \R)$. Then
\begin{equation}\label{E_trace}
\lim_{\delta \to 0} \int_{t^-}^{t^+} \frac{1}{\delta}\int_{\gamma_x(t)}^{\gamma_x(t)+\delta}|u(t,x)-u(t,\bar\gamma_x(t))|dxdt = 0.
\end{equation}
\end{lemma}
\begin{proof}
For $t \in (t^-,t^+)$ we set
\begin{equation*}
v_\delta(t):= \frac{1}{\delta}\int_0^\delta |u(t,\bar \gamma_x(t)+y)-u(t,\bar \gamma_x(t))|dy,
\end{equation*}
so that \eqref{E_trace} is equivalent to $v_\delta$ to converge to 0 in $L^1(t^-,t^+)$. For every $t \in [t^-+\delta,t^+-\delta]$ we have
\begin{equation}\label{E_trace2}
|u(t,\bar \gamma_x(t)+y)-u(t,\bar \gamma(t))| \le \frac{1}{2\delta}\int_{t-\delta}^{t+\delta}\left( |u(t,\bar \gamma_x(t)+y)-u(t',\bar \gamma_x(t'))| + 
|u(t',\bar \gamma_x(t'))- u(t,\bar \gamma_x(t))| \right) dt'.
\end{equation}
Since $|v_\delta|\le 1$, integrating \eqref{E_trace2} with respect to $t$ we get
\begin{equation}\label{E_trace3}
\begin{split}
\int_{t^-}^{t^+}v_\delta(t)dt \le &~ 2\delta + \int_{t^-+\delta}^{t^+-\delta} \frac{1}{2\delta^2} \int_0^\delta \int_{t-\delta}^{t+\delta}  |u(t,\bar \gamma_x(t)+y)-u(t',\bar \gamma_x(t'))| dt'dy dt \\
&~ +  \int_{t^-+\delta}^{t^+-\delta} \frac{1}{2\delta} \int_{t-\delta}^{t+\delta} |u(t',\bar \gamma_x(t'))- u(t,\bar \gamma_x(t))|dt'dt.
\end{split}
\end{equation}
By Lebesgue differentiation theorem, $\frac{1}{2\delta} \int_{t-\delta}^{t+\delta} |u(t',\bar \gamma_x(t'))- u(t,\bar \gamma(t))|dt'\to 0$ as $\delta \to 0$ for $\L^1$-a.e. $t\in (t^-,t^+)$ therefore
the last term in \eqref{E_trace3} converges to zero by dominated convergence theorem.
By applying Tonelli theorem to the second term in the right-hand side of  \eqref{E_trace3} we get
\begin{equation}\label{E_trace4}
\begin{split}
\int_{t^-+\delta}^{t^+-\delta} \frac{1}{2\delta^2} \int_0^\delta \int_{t'-\delta}^{t'+\delta} & |u(t,\bar \gamma_x(t)+y)-u(t',\bar \gamma_x(t'))| dtdy dt' \le \\
& \int_{t^-+\delta}^{t^+-\delta} \frac{1}{2\delta^2} \int_{B_{(2+L)\delta}(t', \bar \gamma_x(t'))}|u(z)-u(t',\bar \gamma_x(t'))|dz dt', 
\end{split}
\end{equation}
where $L$ is the Lipschitz constant of $\bar \gamma_x$ and $B_{(2+L)\delta}(t', \bar \gamma_x(t'))$ denotes the ball in $\R^2$ of radius $(2+L)\delta$ 
and center $(t', \bar \gamma_x(t'))$. Since by assumption $\L^1$-a.e. $t \in (t^-,t^+)$ is a Lebesgue point of $u$, the right-hand side of \eqref{E_trace4} converges to 0 by the 
dominated convergence theorem and this concludes the proof.
\end{proof}

The following proposition formalizes the intuition that a curve lying in the hypograph of a function cannot cross from above a curve lying in the epigraph of the same function.
\begin{proposition}\label{P_h-e}
Let $\bar \gamma \in \Gamma_h, G\subset \Gamma_e$ and $\bar t>0$ be such that for every $\gamma\in G$ it holds $\bar \gamma_x(\bar t) < \gamma_x(\bar t)$.
Then for every $t \in [\bar t,T]$ it holds
\begin{equation*}
\omega_e(\{\gamma\in G: \gamma_x(t)<\bar \gamma_x(t)\})=0.
\end{equation*}
\end{proposition}
\begin{proof}
Given $\delta>0$ we define $\phi_\delta:\R\to \R$ by
\begin{equation*}
\phi_\delta (x)=
\begin{cases}
1 & \mbox{if }x\le 0, \\
\frac{1}{\delta}(\delta - x) & \mbox{if }x \in (0,\delta), \\
0 & \mbox{if }x\ge\delta,
\end{cases}
\end{equation*}
and for every $t \in [\bar t, T]$ we consider
\begin{equation*}
\Psi_\delta (t) = \int_{G}\phi_\delta(\gamma_x(t)-\bar \gamma_x(t))d\omega_e(\gamma).
\end{equation*}
Clearly for every $\delta>0$ and every $t \in [\bar t, T]$ it holds 
\begin{equation}\label{E_est_delta}
\omega_e(\{\gamma^2\in G: \gamma^2_x(t)<\gamma^1_x(t)\})\le \Psi_\delta(t)
\end{equation}
and by assumption
\begin{equation*}
\Psi_\delta (\bar t) \le \omega_e(\gamma \in G: \gamma_x(\bar t) \in [\bar\gamma_x(\bar t),\bar\gamma_x(\bar t)+\delta])\le \delta.
\end{equation*}
Since $\omega_e$ is concentrated on curves satisfying \eqref{E_characteristic}, for every $t\in [\bar t, T]$
\begin{equation*}
\Psi_\delta(t) = \Psi_\delta(\bar t) + \int_{\bar t}^t \frac{1}{\delta}\int_{G(\delta,t')}(\bar\gamma_v(t')-\gamma_v(t'))d\omega_e(\gamma)dt'\le \delta + \int_{\bar t}^t \frac{1}{\delta}\int_{G(\delta,t')}(\bar\gamma_v(t')-\gamma_v(t'))^+d\omega_e(\gamma)dt',
\end{equation*}
where $G(\delta, t'):=\{\gamma \in G: \gamma_x(t')\in [\bar \gamma_x(t'),\bar \gamma_x(t')+\delta]\}$.
Since $\bar \gamma \in \Gamma_h$ for $\L^1$-a.e. $t'\in (t^-,t^+)$ it holds $\bar \gamma_v(t')< u(t',\bar\gamma_x(t'))$. Moreover for $\omega_e$-a.e. $\gamma \in \Gamma$ and every $t \in [0,T]$ it holds $\gamma_v(t)\in [0,1]$, therefore 
\begin{equation*}
\begin{split}
\Psi_\delta(t) \le &~ \delta + \int_{\bar t}^t \frac{1}{\delta}\int_{G(\delta,t')}(u(t', \bar \gamma_x(t'))-\gamma_v(t'))^+d\omega_e(\gamma)dt' \\
\le &~ \delta + \int_{\bar t}^t  \frac{1}{\delta}\left( \omega_e(\{\gamma \in G(\delta,t'): \gamma_v(t') < u(t',\bar \gamma_x(t')) \}) \right)dt'\\
\le &~  \delta + \int_{\bar t}^t \frac{1}{\delta}\int_0^\delta |u(t',\gamma_x(t'))-u(t',\gamma_x(t')+y)|dy dt',
\end{split}
\end{equation*}
where the last inequality follows by \eqref{E_repr_e}. The claim follows by Lemma \ref{L_trace} and \eqref{E_est_delta} by letting $\delta \to 0$.
\end{proof}

For every $(\bar t, \bar x)\in [0,T]\times \R$ we denote by 
\begin{equation}\label{E_def_G}
G_{\bar t, \bar x}^l:= \{\gamma\in \Gamma_h: \gamma_x(\bar t)<\bar x\}, \qquad G_{\bar t, \bar x}^r:= \{\gamma\in \Gamma_e: \gamma_x(\bar t)>\bar x\}.
\end{equation}

\begin{corollary}\label{C_no_crossing}
Let $(\bar t, \bar x) \in [0,T)\times \R$ and consider $G_{\bar t, \bar x}^l, G_{\bar t, \bar x}^r$ as above. Then there exists a Lipschitz function $f_{\bar t, \bar x}:[\bar t, T] \to \R$ such that 
 such that for every $t \in [\bar t, T]$ it holds
\begin{equation}\label{E_separation}
\omega_h(\{\gamma \in G_{\bar t, \bar x}^l: \gamma_x(t)>f_{\bar t,\bar x}(t)\})=0 \qquad \mbox{and} \qquad \omega_e(\{\gamma\in G_{\bar t, \bar x}^r: \gamma_x(t) <f_{\bar t,\bar x}(t)\})=0.
\end{equation}
\end{corollary}
\begin{proof}
If $G_{\bar t,\bar x}^l=\emptyset$ we set $f_{\bar t,\bar x}(t)= \bar x$ and \eqref{E_separation} follows by \eqref{E_characteristic} since for $\omega_e$-a.e. $\gamma\in \Gamma$, it holds $\dot\gamma_x(t)\ge 0$ for $\L^1$-a.e. $t\in (0,T)$. If $G_{\bar t,\bar x}^l \ne \emptyset$, then we set
\begin{equation}
f_{\bar t, \bar x} = \sup_{\gamma \in G_{\bar t, \bar x}^l} \gamma_x.
\end{equation}
The first condition in \eqref{E_separation} is trivially satisfied and by Proposition \ref{P_h-e} for every $\bar \gamma \in G_{\bar t, \bar x}^l$ it holds
\begin{equation}\label{E_est_gamma}
\omega_e(\{\gamma \in G_{\bar t, \bar x}^r:  \gamma_x(t) <\bar \gamma_x(t)\})=0.
\end{equation}
Let $\{t_i\}_{i\in \N}$ be an enumeration of $\Q \cap (\bar t, T)$ and for every $i,n\in \N$ let $\gamma^{i,n}\in G_{\bar t, \bar x}^l$ be such that $\gamma^{i,n}_x(t_j)\ge f_{\bar t,\bar x}(t_i) - \frac{1}{n}$. 
For every $\gamma \in G_{\bar t, \bar x}^l$ the $x$-component $\gamma_x$ is $1$-Lipschitz, hence $f_{\bar t,\bar x} = \sup_{i,n\in \N}\gamma^{i,n}_x$.
Therefore it follows by \eqref{E_est_gamma} that for every $t \in [\bar t, T]$ it holds
\begin{equation}
\omega_e(\{\gamma\in G_{\bar t, \bar x}^r: \gamma_x(t) <f_{\bar t,\bar x}(t)\}) = \omega_e\left( \bigcup_{i,n\in \N}\left\{ \gamma\in G_{\bar t, \bar x}^r: \gamma_x(t) < \gamma_x^{i,n}(t) \right\}\right)=0. \qedhere
\end{equation}
\end{proof}

Thanks to Corollary \ref{C_no_crossing} we can identify the candidate rectifiable set $J^-$ on which $(p_{t,x})_\sharp \mu^-$ is concentrated: 
let $(\bar t_i)_{i\in \N}$ be an enumeration of $[0,T]\cap \Q$ and $(\bar x_j)_{j\in \N}$ be an enumeration of $\Q$. Then we set
\begin{equation}\label{E_def_J-}
J^-:= \bigcup_{i,j \in \N} C_{f_{\bar t_i,\bar x_j}}, \qquad \mbox{where} \qquad C_{f_{\bar t,\bar x}}:=\{ (t,x)\in [\bar t, T]\times \R: x=f_{\bar t, \bar x}(t)\}.
\end{equation}

\section{Concentration of the entropy dissipation}
In this section we prove Theorem \ref{T_main} and Theorem \ref{T_one_entropy}.
In the next lemma we couple the two representations $\omega_h,\omega_e$ taking into account that they represent the same measure $\mu^-$ as in \eqref{E_mu-}.
We denote by $X=[0,T]\times \R \times [0,1]$ and consider the measures $\omega_h \otimes \mu^-_\gamma$ and $\omega_e \otimes \mu^+_\gamma$ defined on the set $\Gamma \times X$ by
\begin{equation*}
\omega_h\otimes \mu^-_\gamma (G\times E) = \int_G \mu^-_\gamma(E)d\omega_h(\gamma) \qquad \mbox{and} \qquad 
\omega_e\otimes \mu^+_\gamma (G\times E) = \int_G \mu^+_\gamma(E)d\omega_e(\gamma),
\end{equation*}
for every measurable $E \subset X, G\subset \Gamma$.
\begin{lemma}
Denote by $p_1,p_2: (\Gamma \times X)^2\to \Gamma \times X$ the standard projections. Then there exists a plan $\pi^- \in \M((\Gamma \times X)^2)$ with marginals
\begin{equation}\label{E_marginals}
\begin{split}
(p_1)_{\sharp}\pi^- =   \omega_h\otimes \mu^-_\gamma, \\
(p_2)_\sharp \pi^-=   \omega_e\otimes \mu^+_\gamma,
\end{split}
\end{equation}
concentrated on the set
\begin{equation*}
\begin{split}
\mathcal G := \big\{((\gamma,t,x,v),(\gamma',t',x',v')) \in (\Gamma \times X)^2 : \, &t=t', \gamma_x(t)=x=x'=\gamma'_x(t'), v=v', \\
&v \in [\gamma_v(t+),\gamma_v(t-)]\cap [\gamma'_v(t-),\gamma'_v(t+)]\big\}.
\end{split}
\end{equation*}
\end{lemma}
\begin{proof}
First we observe that by definition, $\omega_h\otimes \mu^-_\gamma$ is concentrated on the set
\begin{equation*}
\G^-_h:= \{(\gamma,t,x,v) \in \Gamma \times X: \gamma_x(t)=x, v \in [\gamma_v(t+),\gamma_v(t-)]\}
\end{equation*}
and $ \omega_e\otimes \mu^+_\gamma$ is concentrated on the set
\begin{equation*}
\G^+_e:= \{(\gamma,t,x,v) \in \Gamma \times X: \gamma_x(t)=x, v \in [\gamma_v(t-),\gamma_v(t+]\}.
\end{equation*}

Denoting by $p_X: \Gamma \times X \to X$ the standard projection it follows from \eqref{E_mu-} that
\begin{equation*}
(p_X)_\sharp (\omega_h \otimes \mu^-_\gamma) = \mu^- = (p_X)_\sharp (\omega_e \otimes \mu^+_\gamma).
\end{equation*}
By the disintegration theorem (see for example \cite{AFP_book}) there exist two measurable families of probability measures 
$(\mu^{-,h}_{t,x,v})_{(t,x,v) \in X}, (\mu^{+,e}_{t,x,v})_{(t,x,v) \in X} \in \P(\Gamma\times X)$ such that
\begin{equation}\label{E_disintegration}
\omega_h \otimes \mu^-_\gamma = \int_X \mu^{-,h}_{t,x,v} d\mu^- \qquad \mbox{and} \qquad \omega_e \otimes \mu^+_\gamma =  \int_X \mu^{+,e}_{t,x,v} d\mu^-
\end{equation}
and for $\mu^-$-a.e. $(t,x,v)$ the measures $\mu^{-,h}_{t,x,v}$ and $\mu^{+,e}_{t,x,v}$ are concentrated on the set 
\begin{equation*}
p_X^{-1}(\{t,x,v\}) = \{(\gamma,t',x',v')\in \Gamma \times X: t'=t,x'=x,v'=v\}.
\end{equation*}
Moreover, since $\omega_h\otimes \mu^-_\gamma$ is concentrated on the set $\G^-_h$ and $\omega_e\otimes \mu^+_\gamma$ is concentrated on the set $\G^+_e$,
we have that for $\mu^-$-a.e. $(t,x,v)$ the measure $\mu^{-,h}_{t,x,v}$ is concentrated on $p_X^{-1}(\{t,x,v\}) \cap \G^-_h$ and $\mu^{+,e}_{t,x,v}$ is concentrated on 
$p_X^{-1}(\{t,x,v\}) \cap \G^+_e$.
We eventually set
\begin{equation*}
\pi^-:= \int_X \left(\mu^{-,h}_{t,x,v} \otimes \mu^{+,e}_{t,x,v}\right) d\mu^-.
\end{equation*}
From \eqref{E_disintegration} it directly follows \eqref{E_marginals} and by the above discussion for $\mu^-$-a.e. $(t,x,v) \in X$ the measure $\mu^{-,h}_{t,x,v} \otimes \mu^{+,e}_{t,x,v}$
is concentrated on $(p_X^{-1}(\{t,x,v\}) \cap \G^-_h) \times (p_X^{-1}(\{t,x,v\}) \cap \G^+_e)$, therefore $\pi^-$ is concentrated on
\begin{equation*}
\bigcup_{(t,x,v) \in X} (p_X^{-1}(\{t,x,v\}) \cap \G^-_h) \times (p_X^{-1}(\{t,x,v\}) \cap \G^+_e) = \G
\end{equation*}
and this concludes the proof.
\end{proof}

The following elementary lemma is about functions of bounded variation of one variable: we refer to \cite{AFP_book} for the theory of BV functions.
\begin{lemma}\label{L_BV}
Let $v: (a,b)\to \R$ be a $\BV$ function and denote by $D^-v$ be the negative part of the measure $Dv$. Then for $\tilde D^-v$-a.e. $\bar x \in (a,b)$ there exists $\delta >0$ such that
\begin{equation*}
\bar v(x) > \bar v(\bar x) \quad \forall x \in (\bar x - \delta, \bar x) \qquad \mbox{and} \qquad \bar v(x) < \bar v(\bar x) \quad \forall x \in (\bar x, \bar x +  \delta).
\end{equation*}
\end{lemma}

The proof of Theorem \ref{T_main} is obtained considering separately the positive and the negative parts of $\mu$; in the following theorem we deal with $\mu^-$. 
\begin{theorem}\label{T_concentration}
The measure $(p_{t,x})_\sharp \mu^-$ is concentrated on the set $J^-$, defined in \eqref{E_def_J-}.
\end{theorem}
\begin{proof}
Step 1. For every $(\bar t, \bar x)\in [0,T)\times \R$ we consider the measure
\begin{equation*}
\pi^-_{\bar t, \bar x} := \pi^- \llcorner \big( G_{\bar t, \bar x}^l \times [\bar t, T] \times \R \times [0,1])\times (G_{\bar t, \bar x}^r \times [\bar t, T] \times \R \times [0,1]) \big),
\end{equation*}
where $G_{\bar t, \bar x}^l, G_{\bar t, \bar x}^r$ are defined in \eqref{E_def_G}
and we set
\begin{equation*}
\begin{split}
p^1_{t,x}: (\Gamma\times [0,T]\times \R\times [0,1])^2 & \to [0,T]\times \R. \\
(\gamma,t,x,v,\gamma',t',x',v') & \mapsto (t,x)
\end{split}
\end{equation*}
We prove that the measure $(p^1_{t,x})_\sharp \pi^-_{\bar t, \bar x}$ is concentrated on the graph of $f_{\bar t,\bar x}$, namely $C_{f_{\bar t,\bar x}}$.

First we observe that 
\begin{equation}\label{E_est_Omega+}
\begin{split}
(p^1_{t,x})_\sharp \pi^-_{\bar t, \bar x} \le &~ (p^1_{t,x})_\sharp \big[\pi^- \llcorner \big( (G^l_{\bar t,\bar x} \times [\bar t, T] \times \R \times [0,1])\times (\Gamma \times X) \big) \big]\\
= &~ \omega_h \llcorner G^l_{\bar t,\bar x} \otimes  \mu^-_\gamma\llcorner([\bar t, T]\times \R \times [0,1]).
\end{split}
\end{equation}
Since for $\omega_h$-a.e. $\gamma \in G^l_{\bar t,\bar x}$ it holds $(t,\gamma_x(t))\notin \Omega_{\bar t, \bar x}^+$ for every $t \in [\bar t, T]$, then for 
$\omega_h$-a.e. $\gamma \in G^l_{\bar t,\bar x}$ it holds $\mu_\gamma^-(\Omega_{\bar t, \bar x}^+\times [0,1])=0$ and therefore it follows by \eqref{E_est_Omega+} that
\begin{equation}\label{E_null+}
(p^1_{t,x})_\sharp \pi^-_{\bar t, \bar x} (\Omega_{\bar t,\bar x}^+)=0.
\end{equation}
In the same way we get
\begin{equation}\label{E_null-}
(p^2_{t,x})_\sharp \pi^-_{\bar t, \bar x}(\Omega_{\bar t, \bar x}^-)=0,
\end{equation}
where
\begin{equation*}
\begin{split}
p^2_{t,x}: (\Gamma\times [0,T]\times \R\times [0,1])^2 & \to [0,T]\times \R \\
(\gamma,t,x,v,\gamma',t',x',v') & \mapsto (t',x').
\end{split}
\end{equation*}
Finally, since $\pi^-$ is concentrated on $\G$, then
\begin{equation*}
(p^1_{t,x} \otimes p^2_{t,x})_\sharp \pi^- \in \M(([0,T]\times \R)^2)
\end{equation*}
is concentrated on the graph of the identity on $[0,T]\times \R$ and in particular $(p^1_{t,x})_\sharp \pi^-_{\bar t, \bar x}= (p^2_{t,x})_\sharp \pi^-_{\bar t, \bar x}$.
Therefore it follows from \eqref{E_null+} and \eqref{E_null-} that $(p^1_{t,x})_\sharp \pi^-_{\bar t, \bar x}$ is concentrated on
\begin{equation*}
([0,T]\times \R) \setminus (\Omega_{\bar t,\bar x}^+ \cup \Omega_{\bar t,\bar x}^-)= C_{f_{\bar t,\bar x}}.
\end{equation*}
Step 2. We prove that for $\pi^-$-a.e. pair $(\gamma,t,x,v,\gamma',t',x',v') \in (\gamma\times X)^2$ there exists $\delta >0$ such that
for every $s \in [t-\delta, t)$ it holds $\gamma_x(s)<\gamma'_x(s)$.

By applying Lemma \ref{L_BV} to the $BV$ function $\gamma_v:(\bar t, T)\to [0,1]$ we get that for $(\tilde D\gamma_v)^-$-a.e. $t \in (\bar t, T)$ there exists
$\delta >0$ such that for every $s\in (t-\delta,t)$ it holds $\gamma_v(s)<\gamma_v(t)$. Moreover for every $t$ such that $\gamma_v$ has a negative jump and for every
$v \in [\gamma_v(t+),\gamma_v(t-))$ there exists $\delta>0$ such that for every $s\in (t-\delta,t)$ it holds $\gamma_v(s)<v$.
Since for every $\gamma \in \Gamma$ the measure $\mu^-_\gamma$ has no atoms and the set
\begin{equation*}
E_\gamma:=\{(t,x,v)\in X: \gamma_v(t-) = v > \gamma_v(t+) \}
\end{equation*}
is countable, then $\mu_\gamma^-(E_\gamma)=0$ and in particular
\begin{equation*}
\omega_h\otimes \mu^-_\gamma \big( \{(\gamma,t,x,v) \in \Gamma \times X: \gamma_v(t-)=v>\gamma_v(t+)\}\big)=0.
\end{equation*}
Therefore it follows by the argument above that for $\omega_h\otimes \mu^-_\gamma$-a.e. $(\gamma,t,x,v)$ there exists $\delta_h>0$ such that for every 
$s \in (t-\delta_h,t)$ it holds
\begin{equation*}
\gamma_v(s)>v \qquad \mbox{and hence} \qquad  \gamma_x(s)<\gamma_x(t)- v (t-s)
\end{equation*}
by \eqref{E_characteristic}.
In the same way we obtain that for $\omega_e\otimes \mu^+_\gamma$-a.e. $(\gamma',t',x',v')$ there exists $\delta_e>0$ such that for every $s \in (t'-\delta_e,t')$ it holds
\begin{equation*}
\gamma_x(s)>\gamma_x(t') -v' (t'-s).
\end{equation*}
Finally, since $\pi^-$ is concentrated on $\G$, then for $\pi^-$-a.e. $(\gamma,t,x,v,\gamma',t,x',v')$ we set $\delta = \min(\delta_h,\delta_e)>0$, so that for every $s \in (t-\delta, t)$ it holds 
\begin{equation*}
\gamma_x(s) < \gamma_x(t) -v (t-s)= x - v (t-s) =x' - v' (t'-s) \gamma'_x(t')-v'(t'-s)< \gamma'_x(s)
\end{equation*}
and this concludes Step 2.

By Step 2 it follows that for $\pi^-$-a.e. $(\gamma,t,x,v,\gamma',t,x',v')$ there exists $(\bar t_i, \bar x_j) \in ([0,T]\cap \Q)\times \Q$ such that 
\begin{equation*}
(\gamma,t,x,v,\gamma',t,x',v') \in  (G_{\bar t_i, \bar x_j}^l \times [\bar t_i, T] \times \R \times [0,1])\times (G_{\bar t_i, \bar x_j}^r \times [\bar t_i, T] \times \R \times [0,1])
\end{equation*}
so that $\pi^-$ is concentrated on
\begin{equation*}
\bigcup_{i,j\in \N} (G_{\bar t_i, \bar x_j}^l \times [\bar t_i, T] \times \R \times [0,1])\times (G_{\bar t_i, \bar x_j}^r \times [\bar t_i, T] \times \R \times [0,1]).
\end{equation*}
Therefore by Step 1 it follows that $(p^1_{t,x})_\sharp \pi^-$ is concentrated on
\begin{equation*}
\bigcup_{i,j\in \N}C_{f_{\bar t_i,\bar x_j}}=J^-.
\end{equation*}
We finally observe that from \eqref{E_marginals} and \eqref{E_mu-} it follows that
\begin{equation*}
(p^1_{t,x})_\sharp \pi^- = (p_{t,x})_\sharp((p_1)_\sharp \pi^-) =  (p_{t,x})_\sharp (\omega_h\otimes \mu^-_\gamma) = (p_{t,x})_\sharp \mu^-
\end{equation*}
and therefore the proof is completed.
\end{proof}

The analogous statement for the positive part of $\mu$ follows from the observation that reversing the direction of time turns entropy solutions into anti-entropic solutions.

\begin{proposition}\label{P_mu+}
The measure $(p_{t,x})_\sharp \mu^+$ is concentrated on a 1-rectifiable set $J^+$.
\end{proposition}
\begin{proof}
We consider the two reflections $R:[0,T]\times \R \to [0,T]\times \R$ and $\bar R:X \to X$ defined by
\begin{equation*}
R(t,x)=(T-t, -x) \qquad \mbox{and} \qquad \bar R(t,x,v)=(T-t,-x,v).
\end{equation*}
Observe that $R^{-1}=R$ and $\bar R^{-1}=\bar R$. From the kinetic formulation for $u$:
\begin{equation*}
\chi (t,x,v):= 
\begin{cases} 
1 & \mbox{if }u(t,x)>v \\
0 & \mbox{otherwise}
\end{cases}, \qquad
\partial_t\chi + v \partial_x \chi = \partial_v \mu
\end{equation*}
we deduce that $\tilde \chi:= \chi \circ \bar R$ satisfies
\begin{equation}\label{E_tilde}
\partial_t \tilde \chi + v \partial_x \tilde \chi = \partial_v \tilde \mu, \qquad \mbox{with} \qquad \tilde \mu = - \bar R_\sharp \mu
\end{equation}
and it is the kinetic formulation of
\begin{equation*}
\tilde u(t,x)= \int_0^1\tilde \chi(t,x,v)dv = (u\circ R) (t,x).
\end{equation*}
In particular $\tilde u$ is a weak solution with finite entropy production to the Burgers equation \eqref{E_Burgers} and by Theorem \ref{T_concentration} 
the measure $(\pi_{t,x})_\sharp \tilde \mu^-$ is concentrated on a 1-rectifiable set $\tilde J^-$.
From \eqref{E_tilde} it follows that $\mu^+ = \bar R_\sharp \tilde \mu^- $ and therefore that $(p_{t,x})_\sharp \mu^+$ is concentrated on $J^+:=R(\tilde J^-)$, which is obviously 1-rectifiable.
\end{proof}

Theorem \ref{T_main} immediately follows from Theorem \ref{T_concentration} and Proposition \ref{P_mu+} setting
\begin{equation*}
J'=J^-\cup J^+.
\end{equation*}

We conclude by proving Theorem \ref{T_one_entropy}, which essentially follows from the observation that every negative shock of $u$ is entropic and every positive shock is anti-entropic.
\begin{proof}
By \eqref{E_all_entropies} it follows that
\begin{equation*}
|(p_{t,x})_\sharp \mu| = \mu_{\bar \eta} \vee \mu_{-\bar \eta},
\end{equation*}
where $\bar \eta (v)=v^2/2$. Therefore in order to prove the theorem, it is sufficient to check that for $\nu$-a.e. $(t,x)\in [0,T]\times \R$ the supremum in \eqref{E_nu} is attained at $\bar \eta$ or at $-\bar \eta$.
Let us denote by $\tilde J^+\subset J$ the set of points for which $u^-<u^+$ and by $\tilde J^-\subset J$ the set of points for which $u^+<u^-$.
In view of Theorem \ref{T_main} and \eqref{E_rect} it holds $\nu (([0,T]\times \R) \setminus (\tilde J^+\cup \tilde J^-))=0$.
We claim that
\begin{equation}\label{E_no_canc}
\nu\llcorner \tilde J^+ = \mu_{\bar \eta} \qquad \mbox{and} \qquad \nu\llcorner \tilde J^- = \mu_{-\bar \eta}.
\end{equation}
We consider the second equality:
by the Rankine-Hugoniot conditions the normal to $J$ in \eqref{E_rect} is determined for $\H^1$-a.e. $(t,x) \in J$ by
\begin{equation*}
\n = (\n_t,\n_x) = \frac{1}{\sqrt{1+\lambda^2}}(\lambda, -1), \qquad \mbox{where} \qquad \lambda = \frac{u^-+u^+}{2}.
\end{equation*}
In particular the density in \eqref{E_rect} takes the form
\begin{equation*}
\begin{split}
\frac{1}{\sqrt{1+\lambda^2}}\big(q(u^-)-q(u^+) - \lambda(\eta(u^-)-\eta(u^+))\big) = &~ \frac{1}{\sqrt{1+\lambda^2}}\int_{u^+}^{u^-}(q'(v) - \lambda \eta'(v))dv \\
= &~  \frac{1}{\sqrt{1+\lambda^2}}\int_{u^+}^{u^-}(v-\lambda)\eta'(v)dv \\
= &~ \left.\frac{(v-\lambda)^2}{2\sqrt{1+\lambda^2}}\right|_{u^+}^{u^-}- \frac{1}{\sqrt{1+\lambda^2}}\int_{u^+}^{u^-}\frac{(v-\lambda)^2}{2} \eta''(v) dv,
\end{split}
\end{equation*} 
which is maximized by $\eta= - \bar \eta$ in the set $\{\eta \in C^2(\R): \|\eta''\|_{C^0}\le 1\}$.
Being the first inequality of \eqref{E_no_canc} completely analogous, this concludes the proof.
\end{proof}

\end{document}